\documentclass[12pt,regno]{amsart}
\usepackage{amsfonts,amsmath,amssymb,amsbsy,amstext,amsthm}
\usepackage{color,enumerate,float,fullpage,verbatim}
\usepackage[shortlabels]{enumitem}
\usepackage{euler, amsfonts, amssymb, latexsym, epsfig,epic}

\usepackage{url}
\usepackage[colorlinks=true,hyperindex, linkcolor=Mahogany, pagebackref=false, citecolor=NavyBlue]{hyperref}
\usepackage[dvipsnames]{xcolor}
\usepackage{verbatim}
\usepackage{color}
\usepackage[all]{xy}  
\usepackage{mathrsfs}

\theoremstyle{definition}
\newtheorem{theorem}{Theorem}[section]
\newtheorem{theoremx}{Theorem}

\numberwithin{equation}{section}

\newtheorem*{theorem*}{Theorem}

\newtheorem{corollary}[theorem]{Corollary}
\newtheorem{lemma}[theorem]{Lemma}

\newtheorem*{claim*}{Claim}

\theoremstyle{definition}
\newtheorem{definition}[theorem]{Definition}

\newtheorem{remark}[theorem]{Remark}

\newtheoremstyle{TheoremNum}
        {8pt}{8pt}              
        {\upshape}                      
        {}                              
        {\bfseries}                     
        {.}                             
        {.5em}                             
        {\thmname{#1}\thmnote{ \bfseries #3}}
  \theoremstyle{TheoremNum}

\newcommand{\m}{\mathfrak{m}}

\newcommand{\ZZ}{\mathbb{Z}}

\newcommand{\FF}{\mathbb{F}}

\newcommand{\IN}{\operatorname{in}}

\newcommand{\reg}{\operatorname{reg}}
\newcommand{\pd}{\operatorname{pd}}

\newcommand{\rank}{\operatorname{rank}}

\newcommand{\Hom}{\operatorname{Hom}}

\newcommand{\Supp}{\operatorname{Supp}}

\newcommand{\lcm}{\operatorname{lcm}}

\newcommand{\Ass}{\operatorname{Ass}}

\newcommand{\depth}{\operatorname{depth}}

\newcommand{\p}{\mathfrak{p}}

\renewcommand{\leq}{\leqslant}
\renewcommand{\geq}{\geqslant}
\newcommand{\gin}{\operatorname{gin}}

\newcommand{\soc}{\operatorname{soc}}

\DeclareMathOperator{\Tor}{Tor}

\title[]{Uniform bounds on projective dimension and Castelnuovo-Mumford regularity}
\author{Giulio Caviglia}
\address{Department of Mathematics, Purdue University, 150 N. University Street, West Lafayette, IN 47907-2067, USA}
\email{gcavigli@purdue.edu}
\author{Alessandro De Stefani}
\address{Dipartimento di Matematica, Universit{\`a} di Genova, Via Dodecaneso 35, 16146 Genova, Italy}
\email{alessandro.destefani@unige.it}


\begin{document}

\begin{abstract}
In this article we obtain uniform effective upper bounds for the projective dimension and the Castelnuovo-Mumford regularity of homogeneous ideals inside a standard graded polynomial ring $S$ over a field. Such bounds are independent of the number of variables of $S$, in the spirit of Stillman's conjecture and of the Ananyan-Hochster's theorem, and depend on partial data extracted from the beginning or the end of the resolution. The main result is an extension of a theorem due to McCullough from 2012. Namely, we bound the projective dimension and the regularity of an ideal in terms of the regularity of a fraction of the syzygies.
\end{abstract} 

\maketitle

\section{Introduction}
Let $k$ be a field, and $S=k[x_1,\ldots,x_N]$ be a polynomial ring over $k$ with the standard grading. Let $I \subseteq S$ be a homogeneous ideal generated by $\mu$ forms of degrees at most $D$. Stillman's conjecture, now a theorem due to Ananyan and Hochster \cite{AH}, predicts the existence of an upper bound for the projective dimension of $S/I$ which only depends on $\mu$ and $D$, and not on $N$. In fact, 
within this framework, $N$ should be thought of as an unknown number, possibly very large compared to the rest of the given data. 

The goal of this paper is to obtain upper bounds on numerical invariants in the spirit of the conjecture of Stillman: given some information on the ideal $I$, independent of $N$, we want to bound effectively the projective dimension and the Castelnuovo-Mumford regularity of $I$. If the input only consists of $\mu$ and $D$, then the best available result is indeed the one due to Ananyan and Hochster \cite{AH}. 
However, such a bound is not explicit except for small values of $\mu$ and $D$.


In our first estimate, we add in input to $\mu$ some information on the first non-vanishing local cohomology module of $S/I$ supported at $\m=(x_1,\ldots,x_N)$. For simplicity we here assume that $\depth(S/I)=0$, and we state our result in full generality in Section \ref{Section main}. 
\begin{theoremx} (see Theorem \ref{thm lc}) \label{THM A}
Let $I \subseteq S$ be a homogeneous ideal generated by $\mu$ forms. If $S/I$ has a non-zero socle element of degree $\alpha$, then $\pd(S/I) \leq \mu^{2^{\alpha}}$. 
\end{theoremx}

Theorem \ref{THM A} is rather surprising to us. For instance, it implies that if $S/I$ has depth zero and $\reg(H^0_\m(S/I))$ is ``small'' compared to $\pd(S/I)$, then $I$ must be generated by ``many'' elements. We also note that no knowledge of the degrees of all the generators of $I$ is needed. 

As a consequence of this result, we obtain an upper bound on the projective dimension of unmixed radical ideals only in terms of $\alpha$ and the height (see Theorem \ref{thm prime}). 

In comparison with \cite{AH} we stress the fact that, while we add additional information on the socle of $S/I$, the actual estimate is not only explicit, but also comparable with well-known bounds of similar kind \cite{CS,Chardin,McP}. 

The first author showed that bounding the projective dimension of $S/I$ or its Castelnuovo-Mumford regularity only in terms of $\mu$ and $D$ are two equivalent problems \cite[Theorem 29.5]{PeevaGradedSyz}. Indeed, in Corollary \ref{coroll reg} we obtain a similar type of bound for the Castelnuovo-Mumford regularity. Regarding this invariant, we recall that McCullough and Peeva showed in \cite{McP} that not only the Eisenbud-Goto conjecture \cite{EG} is false, but that there is no polynomial upper bound to the regularity only in terms of the multiplicity. See also \cite{CCMcPV}.

Next, we obtain a double exponential upper bound on $\pd(S/I)$ by providing as input $\mu, D$, and an estimate on the generators of the second module of syzygies of linear sections of $S/I$. In this direction, we recall a question raised by Craig Huneke: ``Is there a reasonable upper bound for the projective dimension of $S/I$ just in terms of $\mu, D$ and the maximal degree of a minimal generator of the second module of syzygies of $S/I$?'' We cannot answer Huneke's question because in our assumptions we need to take some linear sections of $S/I$, but our result gives some progress towards that.

\begin{theoremx} (see Theorem \ref{thm syz}) \label{THM B}
Assume that $k$ is infinite, and let $I \subseteq S=k[x_1,\ldots,x_N]$ be an ideal generated by $\mu$ homogeneous elements of degree at most $D$. Assume that $\depth(S/I) = 0$. If, for a sufficiently general linear form $\ell$, one has $\reg\left(\Tor_2\left(\frac{S}{I+(\ell)},k\right)\right) \leq C$, then $\pd(S/I) \leq \mu^{2^{\max\{C,D\}-1}}$. 
\end{theoremx}
Again, for simplicity in Theorem \ref{THM B} we assumed that $\depth(S/I) = 0$. Our result in Section \ref{Section main} is more general and does not have this requirement; however, one needs to control generators of second syzygies of more than one linear section of $S/I$. Moreover, by means of standard techniques, both Theorem \ref{THM A} and \ref{THM B} yield upper bounds for the Castelnuovo-Mumford regularity of $I$ (see Corollaries \ref{coroll reg} and \ref{corollary regB}).

Now let $t_i = \reg(\Tor_i^S(S/I,k))$ for $i \in \ZZ_{\geq 0}$. We recall a result of McCullough.
\begin{theorem} \cite[Theorem 4.7]{Jason} Let $I \subseteq S=k[x_1,\ldots,x_N]$ be a homogeneous ideal, and set $c=\lceil \frac{N}{2} \rceil$. Then $\reg(S/I) \leq \sum_{i=1}^c t_i + \frac{\prod_{i=1}^c  t_i}{(c-1)!}$.
\end{theorem} 

While McCullough's theorem provides an upper bound on the regularity of $S/I$ in terms of the generating degrees of some syzygies, as Corollary \ref{corollary regB} does, it does not fit into the framework of uniform bounds as the rest of this paper. In fact, it requires knowledge of $c$, which is essentially equivalent to requiring knowledge of $N$.

Our extension of McCullough's result is twofold: first of all we pass from  the need to control $\lceil \frac{N}{2} \rceil$ syzygies of $S/I$ to only requiring $\lceil \frac{N}{r} \rceil$ of them for any positive integer $r$. Moreover, we only need to know the value of $r$, not of the ratio $\lceil \frac{N}{r} \rceil$ itself. Specifically, given $r \in \ZZ_{>0}$ we let
\[
\reg_{\frac{1}{r}}(S/I) = \sup\left\{t_i - i \ \bigg| \  0 \leq i \leq \left\lceil \frac{N}{r} \right\rceil\right\}.
\]
Our result is the following:
\begin{theoremx} (see Theorem \ref{thm Jason} and Corollary \ref{coroll Jason}) \label{THM C}
Let $I \subseteq S = k[x_1,\ldots,x_N]$ be a homogeneous ideal generated by $\mu \geq 2$ forms, and $r \in \ZZ_{>0}$. If $\reg_{\frac{1}{r}}(S/I) \leq \delta$, then 
\[
\pd(S/I) \leq r \cdot \mu^{2^{\delta}} \text{ and } \reg(S/I) \leq (\delta+1)^{2^{\left(r\mu^{2^{\delta}}-2\right)}}.
\]
\end{theoremx}

\subsection*{Acknowledgments} We thank the anonymous referee for providing useful suggestions. The first named author was partially supported by a grant from the Simons Foundation (SFI-MPS-TSM-00013569, G.C.). The second named author was partially supported by the MIUR Excellence Department Project CUP D33C23001110001, PRIN 2022 Project 2022K48YYP, and by INdAM-GNSAGA.

\section{Preliminaries}
Throughout this article, $k$ is a field, and $S=k[x_1,\ldots,x_N]$ is a graded polynomial ring, with $\deg(x_i)=1$ for every $i=1,\ldots,N$. We will refer to this as the standard grading on $S$. We let $\m=(x_1,\ldots,x_N)$ be the ideal of $S$ generated by elements of positive degree. 
If $M = \bigoplus_{j\in \ZZ} M_j$ is a finitely generated $\ZZ$-graded $S$-module, we let $M_{\leq j}$ be the graded $S$-module generated by $\bigoplus_{p \leq j} M_p$. We let $\beta_{i,j}(M) = \dim_k(\Tor_i^S(M,k)_j)$ be the $(i,j)$-th graded Betti number of $M$ as an $S$-module, and $t_i(M) = \sup\{j \mid \beta_{i,j}(M) \ne 0\}$. We also let $\beta_{i,\leq j}(M) = \sum_{p \leq j}\beta_{i,p}(M)$ and $\beta_i(M) = \sum_{j \in \ZZ} \beta_{i,j}(M)$. The {\it socle of $M$} is $\soc(M) = 0:_M \m$. Recall that $\soc(M)$ is a finitely generated $k$-vector subspace of $H^0_\m(M)$, which is non-zero if and only if $\depth(M)=0$. The {\it projective dimension of $M$} is defined as
\[
\pd(M) = \sup\{i \in \ZZ_{\geq 0} \mid \beta_{i,j}(M) \ne 0 \text{ for some } j\in \ZZ\}
\]
and the {\it Castelnuovo-Mumford regularity} is
\[
\reg(M)  = \sup \{t_i(M)-i \mid i \in \ZZ_{\geq 0} \}.
\]



Now let $\preccurlyeq$ be a monomial order on $S$. Given a set of polynomials $f_1,\ldots,f_s \in S$ as input, which we may assume being monic, we will call an {\it iteration of Buchberger's algorithm} the following procedure:
\begin{enumerate}
\item Compute the $S$-polynomials $S_{ij}$ between $f_i$ and $f_j$ for all $1 \leq i < j \leq s$:
\[
S_{ij} = \frac{\IN(f_j)}{\gcd\left(\IN(f_i),\IN(f_j)\right)} f_i - \frac{\IN(f_i)}{\gcd\left(\IN(f_i),\IN(f_j)\right)} f_j.
\]
\item Perform a division algorithm to get a standard expression of each $S_{ij}$ in terms of $f_1,\ldots,f_s$. That is, write 
\[
S_{ij} = \sum_{t=1}^s g_tf_t + r_{ij},
\]
with $\IN(g_tf_t) \preccurlyeq \IN(S_{ij})$ for all $t$ and either $r_{ij}=0$ or $\Supp(r_{ij}) \cap (\IN(f_1),\ldots,\IN(f_s)) = \emptyset$.
\end{enumerate}
The output are the original polynomials $f_1,\ldots,f_s$, together with the remainders $r_{ij}$ that are not zero, rescaled so that they are themselves monic.

It is well-known that, given an ideal $I=(f_1,\ldots,f_s)$, one can obtain a Gr{\"o}bner basis of $I$ for the monomial order $\preccurlyeq$ after a finite number of iterations of this algorithm, starting from $f_1,\ldots,f_s$ and taking as next input the output of the previous iteration.

The following remark, even if rather straightforward, will be used several times in the rest of the article. For a more general statement, which holds for orders induced by weights and controls also the first syzygies of the initial ideal of $I$, see \cite[Proposition 2.4]{CDS_Prime}. 

\begin{remark} \label{remark deg}
Assume that $I$ and $f_1,\ldots,f_\mu$ are homogeneous, and let $\delta$ be a positive integer. In order to obtain a set of polynomials whose initial forms generate $\IN_{\preccurlyeq}(I)$ up to degree $\delta$, one only needs to perform $\delta-1$ iterations of Buchberger's algorithm. Indeed, any subsequent iteration of the algorithm will necessarily produce $S$-polynomials which, in degree $\leq \delta$, were already obtained at a previous step and thus reduce to zero in (2). Moreover, as starting data one can consider only those polynomials among $f_1,\ldots,f_\mu$ which have degree at most $\delta$. Note that, if the input are $\mu$ polynomials, then the output of one iteration are at most $\binom{\mu}{2}+\mu \leq \mu^2$ polynomials. It follows that, if $I$ can be generated by $\mu$ elements of degree at most $\delta$, then $\IN_{\preccurlyeq}(I)$ has at most $\mu^{2^{\delta-1}}$ minimal generators of degree at most $\delta$.
\end{remark} 

\begin{definition} An ideal $I \subseteq S=k[x_1,\ldots,x_N]$ is said to be {\it Borel-fixed} if it is invariant under the action which sends each $x_j$ to $g \cdot x_j = \sum_{i} g_{ij}x_i$ for any $g=(g_{ij}) \in{\rm GL}_n(k)$ upper triangular matrix.
\end{definition}

In characteristic zero, Borel-fixed monomial ideals coincide with strongly stable ideals (e.g., see \cite[Section 4.2.2]{HerzogHibi}). In positive characteristic, however, the class of Borel-fixed monomial ideals is strictly larger; for instance, $I=(x_1^p,x_2^p)$ in $S=\FF_p[x_1,x_2]$ is Borel-fixed but not strongly stable. If $\preccurlyeq$ is a monomial order, and $I \subseteq S$ is homogeneous, then its generic initial ideal $\gin_{\preccurlyeq}(I)$ is Borel-fixed \cite{Galligo2,BS2}. 

We now recall the Taylor resolution of a monomial ideal (see for instance \cite[17.11]{Eisenbud}). Let $u_1,\ldots,u_r$ be monomials in $S$. For $\Lambda \subseteq \{1,\ldots,r\}$ we let $u_\Lambda = \lcm(u_i \mid i \in \Lambda)$. If $a_\Lambda$ is the exponent vector of $u_\Lambda$, and $S[-a_\Lambda]$ denotes a free cyclic $\ZZ^N$-graded $S$-module with generator in multi-degree $a_\Lambda$, then for $0 \leq i \leq r$ we let 
\[
F_i = \bigoplus_{\substack{\Lambda \subseteq \{1,\ldots,r\} \\|\Lambda|=i}} S[-a_\Lambda].
\]
If $\{e_{\Lambda}\}_{\Lambda \subseteq \{1,\ldots,r\}}$ denotes a graded free basis of $\bigoplus_{i=0}^r F_i$, we define differentials $d_i:F_i \to F_{i-1}$ as
\[
d_i(e_{\Lambda}) = \sum_{j \in \Lambda} {\rm sign}(j,\Lambda) \frac{u_{\Lambda}}{u_{\Lambda \smallsetminus \{j\}}} e_{\Lambda \smallsetminus \{j\}},
\]
where ${\rm sign}(j,\Lambda)$ denotes $(-1)^{s+1}$ if $j$ is the $s$-th element of $\Lambda$ in the natural order. This construction provides a free resolution of $S/I$, where $I=(u_1,\ldots,u_r)$. 
\begin{remark} \label{rem taylor} Note that, if $J$ is any monomial ideal generated by at most $r$ elements, then $\beta_i(S/J) \leq \rank(F_i) = \binom{r}{i}$. In particular, if $\pd(S/J) \geq s$, then $\binom{r}{s} \ne 0$, and therefore $s \leq r$.
\end{remark}  
\begin{definition} Let $M$ be a finitely generated $\ZZ$-graded $S$-module. A homogeneous element $\ell \in S$ is called {\it filter regular for $M$} if $0:_M \ell = \{\eta \in M \mid \eta \ell = 0\}$ has finite length. A sequence of homogeneous elements $\ell_1,\ldots,\ell_r$ is called a {\it filter regular sequence for $M$} if $\ell_{i+1}$ is a filter regular element for $M/(\ell_1,\ldots,\ell_i)M$ for all $0 \leq i \leq r-1$.
\end{definition}
Equivalently, we have that $\ell$ is filter regular for $M$ if $\ell \notin \p$ for all $\p \in \left(\Ass_S(M) \smallsetminus \{\m\}\right)$. Clearly, any $M$-regular element is filter regular; moreover, we note that if $k$ is infinite then any sufficiently general linear form is filter regular for $M$. We note that, if $\ell$ is filter regular for $M$ but not $M$-regular, then necessarily $\depth(M)=0$.

\section{Uniform upper bounds}  \label{Section main}

Let $S=k[x_1,\ldots,x_N]$, and $M$ be a finitely generated $\ZZ$-graded $S$-module. We recall that $\depth(M) = \inf\{r \mid H^r_\m(M) \ne 0\}$, where $\m=(x_1,\ldots,x_N)$. If $\depth(M) = r$, we let $\alpha(M) = \min\{i+r \mid [\soc(H^r_\m(M))]_i \ne 0\}$. 

\begin{lemma} \label{lemma reduction} Let $M$ be a finitely generated $\ZZ$-graded $S$-module, and assume that $r=\depth(M)>0$. If $\ell \in S$ is a linear form which is $M$-regular, then $\alpha(M/\ell M) = \alpha(M)$.  
\end{lemma}
\begin{proof}
Applying local cohomology to the short exact sequence 
\[
\xymatrix{
0 \ar[r] & \displaystyle M[-1] \ar[r]^-{\cdot \ell} & \displaystyle M \ar[r] & \displaystyle \frac{M}{\ell M} \ar[r] & 0
}
\]
gives a long exact sequence 
\[
\xymatrix{
\ldots \ar[r] &  \displaystyle H^{r-1}_\m\left(M\right) \ar[r] &  \displaystyle H^{r-1}_\m\left(\frac{M}{\ell M}\right) \ar[r] &  \displaystyle H^r_\m\left(M\right)[-1] \ar[r]^-{\cdot \ell} & \displaystyle H^r_\m\left(M\right) \ar[r] & \ldots
}
\]
First note that $H^{r-1}_\m(M)=0$. Applying the functor $\Hom_S(k,-)$ gives an exact sequence on socles:
\[
\xymatrix{
0 \ar[r] &  \displaystyle \soc\left(H^{r-1}_\m\left(\frac{M}{\ell M}\right)\right) \ar[r] &  \displaystyle \soc\left(H^r_\m\left(M\right)\right)[-1] \ar[r]^-{\cdot \ell} & \displaystyle \soc\left(H^r_\m\left(M\right)\right) \ar[r] & \ldots
}
\]
Since the multiplication by $\ell$ is the zero map on socles, we get graded isomorphisms
\[
\soc\left(H^{r-1}_\m\left(\frac{M}{\ell M}\right)\right)_n \cong \soc\left(H^r_\m\left(M\right)\right)_{n-1}
\]
for all $n \in \ZZ$. In particular, this yields the desired equality.
\end{proof}

\begin{remark} If we let $p=N-r$ be the projective dimension of $M$, then one can check that $\alpha(M) = \min\{j -p \mid \beta_{p,j}(M) \ne 0\}$. We note that, from this alternative description of $\alpha(M)$, one can get another proof of Lemma \ref{lemma reduction}.
\end{remark}

\begin{theorem} \label{thm lc}
Let $I \subseteq S$ be an ideal generated by $\mu$ homogeneous elements, and let $\alpha=\alpha(S/I)$. Then
\[
\pd(S/I) \leq \mu^{2^{\alpha}}.
\]
\end{theorem}
\begin{proof}
All the invariants involved are not affected by extending the base field; thus, we may assume that $k$ is infinite. If $r=\depth(S/I)>0$, we can then find a regular sequence $\ell_1,\ldots,\ell_r$ for $S/I$ consisting of linear forms. By Lemma \ref{lemma reduction} we may then assume that $\depth(S/I) = 0$, and seek an upper bound for $N = \pd(S/I)$. Our assumption guarantees that $\beta_{N,N+\alpha}(S/I) \ne 0$. Let $\preccurlyeq$ be a monomial order, and $J=\IN_{\preccurlyeq}(I)$. By upper semi-continuity we have that $\beta_{N,N+\alpha}(S/J) \ne 0$, that is, there exists a monomial $u \in \soc(S/J)$ of degree $\alpha$. If we let $J'=J_{\leq \alpha+1}$, then $u$ still represents a non-zero socle element of $S/J'$. In particular, $\depth(S/J')=0$, that is, $\pd(S/J')=N$. 
Since $I$ is generated by $\mu$ elements, by Remark \ref{remark deg} we have that $\beta_{0,\leq \alpha+1}(J) = \beta_{0,\leq \alpha+1}(J') \leq \mu^{2^\alpha}$. 
By Remark \ref{rem taylor} we conclude that $N \leq \mu^{2^{\alpha}}$, as claimed. 
\end{proof}

As already noted in Remark \ref{remark deg}, in place of $\mu$ we could actually use the minimal number of generators of $I$ of degree at most $\alpha+1$. As a consequence of this observation, and using \cite{CDS_Prime}, we obtain the following estimate for unmixed radical ideals.
\begin{theorem} \label{thm prime} Let $I \subseteq S$ be a homogeneous unmixed radical ideal of height $h$, 
and let $\alpha = \alpha(S/I)$. Then
\[
\pd(S/I) \leq \left(h^{2^{\alpha+3}-3} \right)^{2^{\alpha}}.
\]
\end{theorem}
\begin{proof}
By Theorem 4.2 and Proposition 4.7 in \cite{CDS_Prime} we have that $\beta_{0,\leq \alpha+1}(I) \leq h^{2^{\alpha+3}-3}$. We conclude by Theorem \ref{thm lc}.
\end{proof}

Any bound on the projective dimension of an ideal, together with the knowledge of the number and degrees of its generators, allows to obtain a double exponential bound on the Castelnuovo-Mumford regularity, \cite[Theorem 29.5]{PeevaGradedSyz}.

\begin{corollary} \label{coroll reg}
Let $I \subseteq S$ be an ideal generated by $\mu \geq 2$ homogeneous elements of degree at least $2$ and at most $D$. If we let $\alpha = \alpha(S/I)$, then
\[
\reg(I) \leq D^{2^{\left(\mu^{2^{\alpha}}-2\right)}}.
\]
\end{corollary}
\begin{proof}
After possibly enlarging the base field, we can go modulo a maximal regular sequence consisting of linear forms, and we may assume that $\depth(S/I)=0$. By Theorem \ref{thm lc}, we may assume that $N = \pd(S/I) = \mu^{2^\alpha}$. Our assumptions guarantee that $N \geq 3$, and the claimed inequality now follows from \cite[Corollary 2.13]{CDS_Lin}.
\end{proof}

Theorem \ref{thm lc} allows us to recover the other implication of \cite[Theorem 29.5]{PeevaGradedSyz} as well: a bound on the Castelnuovo-Mumford regularity in terms of the number and the degrees of generators of an ideal provides one for its projective dimension, as proved by the first named author.

\begin{corollary} Let $I \subseteq S$ be a homogeneous ideal generated by $\mu$ elements. Then 
\[
\pd(S/I) \leq \mu^{2^{\reg(S/I)}}.
\]
\end{corollary}
\begin{proof}
Let $r=\depth(S/I)$. It suffices to observe that
\[
\alpha(S/I) \leq \max\{i+r \mid [H^r_\m(S/I)]_i \ne 0\} \leq \max\{i+j \mid [H^j_\m(S/I)]_i \ne 0 \} = \reg(S/I),
\]
and combine this with the first bound in Theorem \ref{thm lc}.
\end{proof}

We now turn our attention to our second main result. In addition to information on the minimal number of generators of an ideal $I$ and their degrees, we require knowledge of the degrees of the generators of the first syzygy modules of sufficiently general hyperplane sections. This allows once again to obtain a double exponential bound for the projective dimension of $I$ in terms of the given data.

\begin{theorem} \label{thm syz}
Assume that $k$ is infinite. Let $I \subseteq S$ be a homogeneous ideal generated by $\mu$ forms of degree at most $D$. Let $r=\depth(S/I)$, and assume that $t_2(S/(I+(\ell_1,\ldots,\ell_{r+1}))) \leq C$ for a sufficiently general choice of linear forms $\ell_1,\ldots,\ell_{r+1}$. If we let $\gamma = \max\{C,D\}$, then
\[
\pd(S/I) \leq \mu^{2^{\gamma-1}}.
\]
\end{theorem}
\begin{proof}
After going modulo general linear forms $\ell_1,\ldots,\ell_r$ we can assume that $\pd(S/I) = N$. Moreover, after performing a sufficiently general change of coordinates, we may assume that $J:=\IN_{{\rm revlex}}(I)$ is Borel-fixed and that $\ell_{r+1} = x_N$. By assumption, we then have that $t_2(S/(I+(x_N))) \leq C$. The graded short exact sequence
\[
\xymatrix{
0 \ar[r] & \displaystyle \frac{S}{I:x_N}[-1] \ar[r] & \displaystyle \frac{S}{I} \ar[r] & \displaystyle \frac{S}{I+(x_N)} \ar[r] & 0
}
\]
induces a long exact sequence of $\Tor$ modules
\[
\xymatrix{
\ldots \ar[r] & \displaystyle \Tor_2^S\left(\frac{S}{I+(x_N)},k\right) \ar[r] & \displaystyle \Tor_1^S\left(\frac{S}{I:x_N},k\right)[-1] \ar[r] & \displaystyle \Tor_1^S\left(\frac{S}{I},k\right) \ar[r] & \ldots
}
\]
which, in turn, gives that
\[
\displaystyle t_1\left(\frac{S}{I:x_N}\right)+1 \leq \max\left\{t_2\left(\frac{S}{I+(x_N)}\right),t_1\left(\frac{S}{I}\right)\right\} \leq  \gamma.
\]
Since $I:x_N \ne I$, if we let $j_0 = \min\{j \mid \left[(I:x_N)/I\right]_j \ne 0\}$ then $I:x_N$ must contain a minimal generator of degree $j_0$. As a consequence of the above inequalities we must have $j_0  \leq \gamma-1$. Recall that $\IN_{revlex}(I:x_N) = J:x_N$ (see \cite[15.7]{Eisenbud}). Since passing to initial ideals does not change Hilbert functions, we still have $\left[(J:x_N)/J)\right]_{j_0} \ne 0$, and this implies that there is a monomial $u \in (J:x_N) \smallsetminus J$ of degree $j_0$. 

Now let $J':=J_{\leq \gamma}$. Since the Borel group acts on $S$ preserving degrees, and $J$ is Borel-fixed, we have that $J'$ is Borel-fixed as well. If we let $\ell$ be a filter regular linear element for $J'$, we can perform an upper triangular linear change of coordinates and assume that $\ell=x_N$. Since $J'$ is Borel-fixed, we conclude that $x_N$ is filter regular for $J'$. However, $x_N$ is not $J'$-regular since $ux_N \in J'$ but $u \notin J'$. It follows that $\depth(S/J')=0$, that is, $\pd(S/J')=N$. Now, by Remark \ref{remark deg} we have that $\beta_0(J') = \beta_{0,\leq \gamma}(J') \leq \mu^{2^{\gamma-1}}$, and by Remark \ref{rem taylor} conclude that 
$N \leq \mu^{2^{\gamma-1}}$, as desired.
\end{proof}

Following the same strategy as in Corollary \ref{coroll reg}, we obtain the following upper bound on the regularity.

\begin{corollary} \label{corollary regB} Assume that $k$ is infinite. Let $I \subseteq S$ be a homogeneous ideal generated by $\mu \geq 2$ forms of degree at least $2$ and at most $D$. Let $r=\depth(S/I)$, and assume that $t_2(S/(I+(\ell_1,\ldots,\ell_{r+1}))) \leq C$ for a sufficiently general choice of linear forms $\ell_1,\ldots,\ell_{r+1}$. If we let $\gamma = \max\{C,D\}$, then
\[
\reg(I) \leq D^{2^{\left(\mu^{2^{\gamma-1}}-2\right)}}. 
\]
\end{corollary}

In order to state our last main result we recall some notation from the introduction. For $r \in \ZZ_{>0}$ we let
\[
\reg_{\frac{1}{r}}(S/I) = \sup\left\{\reg\left(\Tor_i^S(S/I,k)\right) - i \ \bigg| \  0 \leq i \leq \left\lceil \frac{\pd(S/I)}{r} \right\rceil\right\}.
\]
Note that $\reg_1(S/I)$ coincides with $\reg(S/I)$, while $\reg_{\frac{1}{2}}(S/I)$ only takes into account the first half of the resolution.
\begin{theorem} \label{thm Jason}
Let $I \subseteq S$ be an ideal  generated by $\mu$ homogeneous elements. Suppose that $\reg_{\frac{1}{r}}(S/I) \leq \delta$ for some $r \in \ZZ_{>0}$. Then
\[
\pd(S/I) \leq r \cdot \mu^{2^{\delta}}.
\]
\end{theorem}
\begin{proof} 
Let $s=\lceil \frac{\pd(S/I)}{r} \rceil -1$. Because $\reg_{\frac{1}{r}}(S/I) \leq \delta$, 
we have that 
\[
\beta_{s+1}(S/I) = \sum_{j \leq \delta} \beta_{s+1,j+s+1}(S/I) = \sum_{j \leq \delta+1} \beta_{s,j+s}(I).
\]
Now fix a monomial order $\preccurlyeq$, and let $J=\IN_{\preccurlyeq}(I)$. By upper semi-continuity we have that $\beta_{s,j}(I) \leq \beta_{s,j}(J)$ for all $j \in \ZZ$, and it follows that
\[
\beta_s(I) = \sum_{j \leq \delta+1} \beta_{s,j+s}(I)  \leq \sum_{j \leq \delta+1} \beta_{s,j+s}(J) \leq \sum_{j \leq \delta+1}  \beta_{s,j+s}(J_{\leq \delta+1}) \leq \beta_{s}(J_{\leq \delta+1}).
\]
The second inequality above is a consequence of the fact that, to compute $\beta_{s,j+s}(J) = \dim_{k}(\Tor_s^S(J,k)_{j+s})$, one only needs to know $J$ in degrees $j-1$ and $j$. In particular, we have that $\pd(S/J_{\leq \delta+1}) \geq \pd(S/I) \geq s+1$. Since $I$ is generated by $\mu$ elements, by Remark \ref{remark deg} we obtain that $\beta_{0,\leq \delta+1}(J) =\beta_{0,\leq \delta+1}(J_{\leq \delta+1}) \leq \mu^{2^{\delta}}$. We conclude by Remark \ref{rem taylor} that 
$s+1 \leq \mu^{2^{\delta}}$, and hence $\pd(S/I) \leq r \cdot (s+1) \leq r \cdot \mu^{2^{\delta}}$.
\end{proof}

\begin{corollary} \label{coroll Jason} Let $I \subseteq S$ be an ideal generated by $\mu \geq 2$ homogeneous elements. Suppose that $\reg_{\frac{1}{r}}(S/I) \leq \delta$ for some $r \in \ZZ_{>0}$. Then
\[
\reg(I) \leq (\delta+1)^{2^{\left(r\mu^{2^{\delta}}-2\right)}}.
\]
\end{corollary}

\end{document}